\documentclass[preprint]{imsart}
\RequirePackage[OT1]{fontenc}
\RequirePackage{amsthm,amsmath,booktabs}
\RequirePackage[authoryear]{natbib}
\RequirePackage[colorlinks,citecolor=blue,urlcolor=blue]{hyperref}

\usepackage{amsthm,amsmath,natbib}

\startlocaldefs

\usepackage{graphics,amssymb,color}
\usepackage{graphicx,subfigure}
\usepackage{dsfont}
\usepackage{algorithm}
\usepackage{algorithmic}

\newcommand{\argmax}{\mathop{\mathrm{argmax}}}
\newcommand{\argmin}{\mathop{\mathrm{argmin}}}
\newcommand{\minimize}{\mathop{\mathrm{minimize}}}

\newcommand{\pseudoY}{{\cal Y}}
\newcommand{\cF}{{\cal F}}
\newcommand{\law}{{\cal L}}

\newcommand{\disp}{\displaystyle}
\newcommand{\real}{\mathbb{R}}
\newcommand{\Pp}{{\mathbb P}}
\newcommand{\Ee}{{\mathbb E}}

\newcommand{\sign}{\mathrm{sign}}
\newcommand{\supp}{\mathrm{supp}}
\newcommand{\diag}{\mathrm{diag}}
\newcommand{\unif}{\mathrm{Unif}}
\newtheorem{theorem}{Theorem}
\newtheorem{lemma}{Lemma}
\newtheorem{corollary}{Corollary}
\newtheorem{remark}{Remark}
\newtheorem{Example}{Example}[section]
\newtheorem{definition}{Definition}
\newtheorem{assumption}{Assumption}

\def\norm#1{\left\|{#1}\right\|} 
\newcommand{\lowerbound}{L_{\E^*}}
\newcommand{\upperbound}{U_{\E^*}}
\newcommand{\lowerboundfix}{L_{\E}}
\newcommand{\upperboundfix}{U_{\E}}

\newcommand{\tlowerstar}{\tilde{L}_{\E^*,\delta}}
\newcommand{\tupperstar}{\tilde{U}_{\E^*,\delta}}
\newcommand{\tlowerdelta}{\tilde{L}_{\E,\delta}}
\newcommand{\tupperdelta}{\tilde{U}_{\E,\delta}}

\newcommand{\loweri}{L_{\E_i}}

\newcommand{\upperi}{U_{\E_i}}
\newcommand{\tloweridelta}{\tilde{L}_{\E_i, \delta}}

\newcommand{\tupperidelta}{\tilde{U}_{\E_i, \delta}}

\newcommand{\hbeta}{\hat{\beta}}

\newcommand{\mD}{\mathcal{D}}

\newcommand{\covdiag}{\Sigma}
\newcommand{\nrow}{r}
\newcommand{\E}{\mathcal{E}}
\newcommand{\wE}{\widetilde{\mathcal{E}}}

\newcommand{\A}{\mathcal{A}}

\newcommand\tf{\widetilde{f}}
\newcommand\tP{\widetilde{P}}

\newcommand{\grad}{\nabla}

\newcommand{\cardS}{|{\cal S}|}

\makeatletter
\newcommand{\pushright}[1]{\ifmeasuring@#1\else\omit\hfill$\displaystyle#1$\fi\ignorespaces}
\newcommand{\pushleft}[1]{\ifmeasuring@#1\else\omit$\displaystyle#1$\hfill\fi\ignorespaces}
\makeatother

\newcommand{\upperboundj}{U_{(j^*,s^*)}}
\newcommand{\lowerboundj}{L_{(j^*,s^*)}}

\endlocaldefs

\begin{document}
\begin{frontmatter}

\title{Asymptotics of selective inference}
\runtitle{Asymptotics of selective inference}

\begin{aug}
\author{\fnms{Xiaoying} \snm{Tian}\corref{}\ead[label=e1]{xtian@stanford.edu},  
\fnms{Jonathan} \snm{Taylor}\ead[label=e2]{jonathan.taylor@stanford.edu}\thanksref{t1}
}
\runauthor{Tian and Taylor}

\affiliation{$^1$Stanford University}

\address{Department of Statistics\\ Stanford University 
\\ Sequoia Hall \\ Stanford, CA 94305, USA \\ \printead{e1} \\
\printead*{e2} } 

\thankstext{t1}{Supported in part by NSF grant DMS 1208857 and
AFOSR grant 113039.}
\end{aug}

\begin{abstract}
In this paper, we seek to establish asymptotic
results for selective inference procedures removing
the assumption of Gaussianity. 
The class of selection procedures we consider
are determined by affine inequalities, which we refer to as
affine selection procedures. Examples of affine selection procedures include
selective inference along the solution path of the LASSO, as well
as selective inference after fitting the LASSO at a fixed value of the regularization parameter.
We also consider some tests in penalized generalized linear models. Our result proves asymptotic
convergence in the high dimensional setting where $n<p$, and $n$ can be of a logarithmic factor
of the dimension $p$ for some procedures. Our method of proof adapts a method of \cite{chatterjee2005simple}.
\end{abstract}



\end{frontmatter}

\section{Introduction}
\label{sec:introduction}

{\em Selective inference} is a recent research topic that studies valid inference after 
a statistical model is suggested by the data \cite{optimal_inference,lee2013exact,spacings,taylor2013tests}. 
Classical inference tools break down at this point as
the data used for the hypothesis test is allowed to be the data used to suggest the hypothesis.
Specifically, instead of being given a priori, the hypothesis to test is dependent on the data, thus random.
Formally, denoted by $\E^*=\E^*(y,X)$ is the model selection procedure, which generates a set of 
hypotheses to test, or perhaps parameters for which to form intervals. 
It is useful to think of $\E^*$ as a point process with values in ${\cal S}$,
where ${\cal S}$ is some collection of questions of possible interest.
Consider the following example,

Suppose $y | X \sim G$ with $y \in \mathbb{R}^n, X \in \mathbb{R}^{n \times p}$, $X$ fixed. 
For any $E \subset \{1, \dots, p\}$ define the functionals 
$$
\beta_{j,E}(G) = e_j^T\argmin_{\beta} \Ee_G(\|y - X_E\beta_E\|^2|X) \quad j \in E,
$$
where $e_j$ is the unit vector with only the $j$-th entry being $1$.
Such functionals $\beta_{j,E}$ is essentially the best linear coefficients within the model 
consisting of only variables in $E$.
Then the collection of possily interesting questions are
$$
{\cal S} = \bigg\{\{\beta_{j,E}, ~j \in E\}: E \subset \{1, \dots, p\}\bigg\}.
$$ 
The data $(y, X)$ will then suggest a subset of interesting variables $E$, 
and $\E^*(y,X)$ designates the target for inference to be $\{\beta_{j,E}, ~ j\in E\}$,
the best linear coefficient within a model consisting of only
the variables in $E$.

Previous literature has studied inference after different model selection procedures $\E^*$.  
Notably \cite{lee2013exact} proposed an exact test within the model suggested by LASSO, that is
$\E^* = \{\beta_{j,E}, ~j \in E\}$, where $E$ is the active set of the LASSO solution. The test
is based upon a pivotal quantity which the authors prove to be distributed as $\unif(0,1)$ if 
the hypothesis to be tested is true. Thus
such quantity $P_j(y)$ can be used to test the hypothesis $H_{0j}: \beta_{j,E} = 0$, 
and control the ``Type-I error'' at level $\alpha$,  
\begin{equation}
\label{eq:typeI}
\Pp\left(P_j(y) \leq \alpha \mid H_{0j} \text{ is true}\right) \leq \alpha.
\end{equation}
By inverting such tests, \cite{lee2013exact} can also construct valid confidence intervals for $\beta_{j,E}$.

It is of course worth noticing that either the hypothesis $H_{0j}$ or the parameters $\beta_{j,E}$
are random as $E$ is suggested by the data. So the ``Type-I error'' \eqref{eq:typeI} is not the
classical Type-I error definition where the hypotheses are given a priori. 
Such inference framework is first considered in \cite{posi}, and we leave the philosophical 
discussions of such approach to \cite{optimal_inference}.

The means by which \cite{lee2013exact} controls the ``Type-I error'' is through constructing 
the p-value functions $P_j$. Such construction is highly dependent on the assumption of normality 
of the error distribution. Other works like \cite{covtest, spacings} used similar approaches.
Compared to these previous work, we seek to remove the Gaussian assumption on the errors
and establish asymptotic distributions of $P_j$ in this work. We state the conditions under which
$P_j$ will be asymptotically distributed as $\unif(0,1)$, and thus $P_j$ can be used as 
p-values to test the hypotheses and asymptotically control the ``Type-I error'' in \eqref{eq:typeI}. 
This allows asymptotically valid inference in the linear regression
setting without normality assumptions. It also allows application of covariance test
\citep{covtest} in generalized linear models.

\subsection{Related works}

\cite{tibshirani2015uniform} also considers uniform convergence of the
statistics proposed by \cite{spacings}, but focuses mainly on the low
dimensional case. In the high dimensional case, they have a negative 
result on the uniform convergence of the pivot. 
In this paper, we instead focus on the high dimensional case and state
the conditions in which the pivot will converge. 
More specifically, $n$ is allowed to be of a logarithmic factor of the
dimension $p$ for two common procedures introduced in Section \ref{sec:example}.

In the works of \cite{belloni2012sparse, pvalue_high_dim, zhang2014confidence,
javanmard2015biasing}, the authors proposed various ways of constructing
confidence regions for the underlying parameters in the high-dimensional
setting. One major difference between these works and our framework is that
they try to achieve full model inference without using the data to choose a
hypothesis. The advantage of such approach is robustness. But in the
high-dimensional setting, with tens of thousands of potential variables, it is
natural to use the data to select hypotheses of interest and perform valid
inference only for those hypotheses.  In addition, some of the full model
inference works require conditions of linear underlying model
\cite{pvalue_high_dim, javanmard2015biasing} which the framework of selective
inference does not require. For more philosophical discussions on the
comparisons of the two approaches, see \cite{optimal_inference}.

\subsection{Organization of the paper}
In Section \ref{sec:affine}, we formally introduce the methods for selective inference with certain model
selection procedures, which we call affine selection procedures. 
In Section \ref{sec:main_results}, we state the main
theorem that will allow asymptotically valid inference. In Section \ref{sec:example},
we will illustrate the applications of our results to two selective inference problems,
selective inference after solving the LASSO at a fixed $\lambda$, and the covariance
test for testing the global null in generalized linear models. We collect all the 
proofs in Section \ref{sec:proofs} and dicuss the directions of future research in 
Section \ref{sec:discussion}. 

%
%

\section{Selective inference with affine selection procedures}
\label{sec:affine}

Suppose we have a design matrix $X \in \real^{n \times p}$, considered fixed, and
\begin{equation}
\label{eq:setup}
y_i |x_i \overset{\mathrm{ind}}{\sim} G(\mu(x_i), \sigma^2(x_i)) 
\end{equation}
where $x_i$ is the $i$-th row of the matrix $X$ and $G(\mu, \sigma^2)$ denotes
any one-dimensional distribution with mean $\mu$ and variance $\sigma^2$. We
also denote $\mu(X) = (\mu(x_1), \dots, \mu(x_n))$ and  
$\Sigma(X) = \diag(\sigma^2(x_i), \dots, \sigma^2(x_n))$, a diagonal matrix
with $\sigma^2(x_i)$ as the diagonal entries. Some feature selection procedure is
then applied on the data to select a subet $E \subseteq \{1, 2, \dots, p\}$ and the
target of inference will be
$\E^*(y, X) = \{\beta_{j, E}, ~ j \in E\}$. In general, we consider certain
selection procedures called the affine selection procedures, 

\begin{definition}[Affine selection procedure]
\label{def:affine}
Suppose a model selection procedure $\E^*: \real^n \times \real^{n \times p} \rightarrow \cal{S}$,
where $\cal{S}$ is a finite set of models, 
$${\cal S}=\{\E_1, \dots, \E_{\cardS}\}.$$
We call $\E^*$ an {\em affine selection procedure}, if the selection event
can be written as an affine set in the first argument of $\E^*$. Formally, 
$\E^*$ is an affine selection procedure
if for each potential model to be selected $\E \in {\cal S}$,
\begin{equation}
\label{eq:affine}
\left\{\E^*(z,X) = \E \right\} = \left\{A(\E,X)z \leq b(\E,X)\right\}, 
\quad (z, X) \in \real^n \times \real^{n \times p}.
\end{equation}
where $A \in \real^{k \times n}$, $b \in \real^k$ and $k \in \mathbb{N}$ are
dependent only on $\E$ and $X$. Moreover, the sets 
$$
\{A(\E_i,X)z \leq b(\E_i,X)\} \subseteq \real^n, \quad i=1, \dots, \cardS
$$ 
are disjoint or their intersections have measure $0$ under the
Lebesgue measure on $\real^n$.
\end{definition}

Examples of affine selection procedures include selection procedures that
are based on $E$, the set of variables chosen by the data and usually some
other information\footnote{See Section \ref{sec:example} for details}.   
Various algorithms can be used to select $E$, e.g. 
$E$ as the active set of the LASSO solution at a fixed $\lambda$
\citep{lee2013exact}, $E$ as the first variable to enter the LASSO or LARS
path \citep{covtest}, (more generally any $\ell_1$ penalized generalized linear
models) or $E$ as the $k$ variables included at the $k$-th step of
forward stepwise selection \citep{spacings}. 

The works of \cite{lee2013exact, covtest, spacings, taylor2013tests} have 
constructed valid p-values when the family $G$ is the Gaussian family.
Formally, the pivotal function depends on the following quantities, 

\subsection{Notations}
 
The pivotal function is determined by the following functions. For any $A \in \real^{k \times n}$, 
$b \in \real^k$, $\Sigma \in \real^{n \times n}$ and $\eta \in \real^n$, we define
\begin{align}
\label{eq:pivot_quantities}
\alpha &= \alpha(A,b,\Sigma,\eta) = \frac{A \Sigma \eta}{\eta^T \Sigma \eta}, \\ 
L(z; A,b,\Sigma, \eta) &= \max_{\alpha_j < 0} \disp\frac{b_j - (Az)_j + 
\alpha_j \eta^Tz}{\alpha_j},\\
U(z; A,b,\Sigma,\eta) &= \min_{\alpha_j > 0} \disp\frac{b_j - (A z)_j + 
\alpha_j \eta^T z}{\alpha_j}.
\end{align}
Furthermore, we define
$$F(x;\sigma^2, m, a, b) = \frac{\Phi((x-m)/\sigma) - 
\Phi((a-m)/\sigma)}{\Phi((b-m)/\sigma) - \Phi((a-m)/\sigma)}$$
which is the CDF of the univariate Gaussian law $N(m,\sigma^2)$ truncated to the interval $[a,b]$.

\subsection{A pivotal quantity with Gaussian errors}

Theorem \ref{thm:general_gaussian} provides the construction of a pivotal function
when the data is normally distributed and $\E^*$ is an affine selection procedure.
We denote the response variables to be $\pseudoY$ when $G$ is the Gaussian
family to distinguish it from $y$ where $G$ is a more general location-scale family. 
Note all distributions in this paper are {\em conditional on $X$}, that is the
law we consider are either $\law(\pseudoY|X)$ or $\law(y|X)$. All random variables 
have access to $X$ as if it were a constant.
\begin{theorem}[\cite{lee2013exact}]
\label{thm:general_gaussian}
Suppose $X \in \real^{n\times p}$ and 
$\pseudoY \sim N(\mu(X), \Sigma(X))$, $\mu(X) \in \real^n, ~\Sigma(X) \in \real^{n \times n}$
and $\E^*$ is an affine selection procedure on $\real^n \times \real^{n \times p}$.
Then any for any  $\eta: \real^n \rightarrow \real^n$ measurable with respect to $\sigma(\E^*)$
we have
\begin{multline}
\label{eq:pivot}
F(\eta(\E^*)^T  \pseudoY; \eta(\E^*)^T \covdiag \eta(\E^*), \eta(\E^*)^T\mu, \lowerbound(\pseudoY), \upperbound(\pseudoY))\bigl \vert\E^*(\pseudoY, X) = \E \\
\sim \unif(0,1), \hspace{20pt}
\end{multline}
where $\E^*(z, X) = \E \iff A(\E, X) z \leq b(\E, X)$ and 
\begin{align}
\lowerboundfix(z) &= L(z, A(\E, X), b(\E, X), \Sigma, \eta(\E)) \label{eq:lower}\\
\upperboundfix(z) &= U(z, A(\E, X), b(\E, X), \Sigma, \eta(\E)). \label{eq:upper}
\end{align}
Moreover, marginalizing over the selection procedure $\E^*$, we have the following
\begin{equation}
\label{eq:pivot_marginal}
F(\eta(\E^*)^T  \pseudoY; \eta(\E^*)^T \covdiag \eta(\E^*), \eta(\E^*)^T\mu, \lowerbound(\pseudoY), \upperbound(\pseudoY))
\sim \unif(0,1). 
\end{equation}
\end{theorem}

The significance of Theorem \ref{thm:general_gaussian} is that assuming the diagonal matrix 
$\Sigma$ is known, the only unknown parameter for the pivotal quantity \eqref{eq:pivot_marginal} is 
$\eta^T \mu$. To test the hypothesis $H_0: \eta^T \mu = 0$, we just need to plug
in the value and then compute \eqref{eq:pivot_marginal}, which then can be 
used as a p-value to accept/reject the hypothesis. For example, if we take
\begin{equation}
\label{eq:contrast}
\eta = X_E(X_E^T X_E)^{-1} e_j, 
\end{equation}
where $e_j$ is the unit vector with only the $j$-th entry being $1$, $\eta^T \mu = \beta_{j, E}$.
The quantity in \eqref{eq:pivot_marginal} is pivotal and can be used to test the hypothesis
$H_{0j}:\beta_{j,E} = 0$, and control the ``Type-I errpr'' \eqref{eq:typeI}.
Since $X$ is fixed, we use the shorthand 
$$
\E^*(z) = \E^*(z, X), ~A(\E) = A(\E, X), ~b(\E) = b(\E, X). 
$$

\section{Asymptotics with non-Gaussian error}
\label{sec:main_results}

Now if we remove the assumption that the error $\law(\pseudoY|X) = N(\mu(X),
\Sigma)$, the conclusion of Theorem \ref{thm:general_gaussian} does not hold
any more.  The best we can hope for is a weak convergence result that the same
pivotal  quantities \eqref{eq:pivot_marginal} would converge to $\unif(0,1)$
(as $n \rightarrow \infty$).  This
requires some conditions on both the distribution $\law(y|X)$ and the selection
procedure $\E^*$.  Our main contribution in this work, Theorem
\ref{thm:pivot_conv} establishes conditions on $\law(y|X)$ and $\E^*$ under
which the pivotal quantity \eqref{eq:pivot_marginal} is asymptotically
distributed as $\unif(0,1)$. 

The main approach is to compare the distribution of the pivots \eqref{eq:pivot_marginal}
under the distribution $\law(y|X)$ with that under Gaussian distribution $\law(\pseudoY|X)$.
In the latter case, the exact distribution is derived in Theorem \ref{thm:general_gaussian}. 
In the following, we establish the conditions where the above two distributions are comparable.

\subsection{Bounding the influence function}
\label{sec:influence}

Note the pivotal quantity in \eqref{eq:pivot_marginal} depends on $y$ either through
the linear functions $\eta^T y$ or the maximum/minimum of linear functions $\lowerbound(y)$, 
$\upperbound(y)$. In approximating the exact Gaussian theory with
asymptotic results a quantity analogous to a Lipschitz constant (in $y$) will be necessary, expressing the
changes in $\eta^T y$ as well as the upper and lower bounds $L_{\E^*}$ and $U_{\E^*}$. This, in some sense,
describes the influence each $y_i$ can have on the pivotal quantity \eqref{eq:pivot_marginal}.

For an affine selection procedure $\E^*: \real^n \times \real^{n\times p} \rightarrow {\cal S}$,
without loss of generality suppose $\E^*$ is surjective. Since $\E^*$ is affine, for any model
$\E \in {\cal S}$, there are the associated $A(\E)$ and $b(\E)$ as defined in \eqref{eq:affine}.  
We define
\begin{equation}
\begin{aligned}
M(\E,\eta) = \max_{\substack{1 \leq i \leq {\tt nrow}(A(\E))\\ 1 \leq j \leq n}}  \left| \frac{A(\E)_{ij}}{(A(\E) \covdiag \eta(\E))_i} \right| +  
 \|\eta(\E)\|_{\infty}, \\
\vspace{1mm}
M( \E^*,\eta) = \max_{\E \in {\cal S}} M(\E,\eta). \hspace{40pt}
\end{aligned}
\end{equation}
We also define
\begin{equation}
\nrow(\E) = {\tt nrow}(A(\E)), \quad
\nrow(\E^*) = \max_{\E \in {\cal S}}{\tt nrow}(A(\E)).
\end{equation}

The quantity $M(\E^*,\eta)$ measures the maximal influence any $y_i$ has on 
a smoothed version of the triple $(\eta(\E^*)^Ty, \lowerbound(y), \upperbound(y))$. As 
$M(\E^*, \eta)$ and $\nrow(\E^*)$ are critical in bounding the difference between
$\law(y|X)$ and $\law(\pseudoY|X)$, it is important to get a sense of their size.
Typically $\nrow(\E^*)$ is less than $p$, and we discuss the typical size
of $M(\E^*, \eta)$ through the following simple example: 

\begin{Example}
\label{eg:independent_col}
Suppose the design matrix $X \in \real^{n\times p}$ is generated in the
following way: we first generate each row independently from a distribution on
$\real^p$, and then normalize the column of $X$ to have length $1$. Suppose
instead of using data to select a model, we just arbitrarily choose a subset
$E$. This is equal to no selection at all, thus $M(\E, \eta) =
\|\eta(\E)\|_{\infty}$. If we want to perform inference for $\beta_{j, E}$, we
take 
$$
\eta = X_E (X_E^T X_E)^{-1} e_j, 
$$
where $e_j$ is the unit vector with only the $j$-th coordinate being $1$.
Since we normalize the columns, it is not hard to verify $(X_E^T X_E)^{-1} = O_p(1)$,
and $\max_{ij}(X_{ij}) = O_p(n^{-1/2})$, thus if the selected variables set always satisfies
$|E| \ll n$, $\eta = O_p(n^{-1/2})$. Therefore $M(\E^*, \eta) = O_p(n^{-1/2})$.
\end{Example}

This is a very simple example which does not involve selection. In reality we
will some meaningful selection procedure that uses the data so $M(\E^*, \eta)$
would involve $A(\E^*)$ and $b(\E^*)$ as well. However, we will see through
examples in Section \ref{sec:example} that it is still reasonable to assume
$M(\E^*, \eta) = O(n^{-1/2})$.

The following theorem compares the distribution of 
$(\eta(\E^*)^Ty, \lowerbound(y), \upperbound(y))$ under $\law(y|X)$ and 
its Gaussian counterpart. 
\begin{theorem}
\label{thm:general}
Fix $X \in \real^{n \times p}$. Suppose 
$(y, \pseudoY)$ are defined conditionally independent given $X$ 
on a common probability space such that 
\begin{itemize}
\item  ${\cal L}(y|X)$  has independent entries
with mean vector $\mu$ and covariance matrix variance $\covdiag$ and finite third moments
bounded by $\gamma$;
\item ${\cal L}(\pseudoY|X) = N(\mu, \covdiag)$.
\end{itemize}

Suppose we are given $\eta \in \sigma(\E^*)$,
then given any bounded function $W \in C^3(\real^3;\real)$
with bounded derivatives satisfying
$$
W(u,v,w) = 
\begin{cases}
    \geq 0  & \text{if } v \leq u \leq w \\
    = 0 & \text{else}
\end{cases}
$$
there exists $N = N(M(\E^*, \eta), \cardS, r(\E^*), W)$, such that the following holds for $n, p \geq N$,
\begin{multline}
\label{eq:general}
\Big|\Ee W\left[\eta(\E^*)^T y, \lowerbound( y),
 \upperbound( y)\right]  
- \Ee W\left[\eta(\E^*)^T  \pseudoY, \lowerbound(\pseudoY), 
\upperbound( \pseudoY)\right] \Big|\\
\leq C(W,\gamma) \bigg[\bigg(\log \big(\nrow(\E^*) \cardS\big)\bigg)^4 n M(\E^*,\eta)^3\bigg]^{\frac{1}{5}} \hspace{30pt}
\end{multline}
where $C(W,\gamma)$ is a constant depending only on the derivatives of $W$ and $\gamma$, and 
$\eta(\E^*)$ is $\eta(\E^*(y))$ or $\eta(\E^*(\pseudoY))$ depending on the context.
\end{theorem}

As it is reasonable to assume $M(\E^*, \eta) = O(n^{-1/2})$, it is reasonable to assume
the RHS of \eqref{eq:general} goes to zero. Thus the distribution of 
$(\eta(\E^*)^T y, \lowerbound(y), \upperbound(y))$ is close to that of 
$(\eta(\E^*)^T \pseudoY, \lowerbound(\pseudoY), \upperbound(\pseudoY))$. In the following,
we discuss the conditions under which the pivotal quantity \eqref{eq:pivot_marginal}
converges. 

\subsection{Smoothness of the pivot}

Note the bound in \eqref{eq:general} also depends on $C(W, \gamma)$, the derivatives of $W$.
Thus besides the influence of each $y_i$ on \eqref{eq:pivot_marginal}, it is also necessary to control
the smoothness of the \eqref{eq:pivot_marginal}. In particular, the pivot in \eqref{eq:pivot_marginal}
takes the form of a truncated Gaussian cdf. Moreover, the smoothness (derivatives) of 
the truncated Gaussian cdf $F(x;\sigma^2,m,a,b)$ can depend heavily on the truncation interval $[a,b]$.
More specifically, a lower bound on the denominator of $F(x;\sigma^2,m,a,b)$ puts some constraints
on the width of the interval $[a,b]$ as well as its distance to the origin.
In our context, $a,b$ corresponds to the upper and lower bounds
appearing in \eqref{eq:pivot_marginal}.
Formally, we assume the following assumption:
\begin{assumption}
\label{A:denominator}
Suppose we have $X_n \in \real^{n \times p_n}$ and $y_n \in \real^n$ is generated according to
\eqref{eq:setup}, and $\pseudoY_n$ is generated independently (conditional on $X_n$) 
from $N(\mu(X_n), \Sigma(X_n))$ a Gaussian distribution with the same means and variances.
We also have affine selection procedures $\E^*=\E^*_n$. We assume there
exists $\delta_n \to 0$ such that
\begin{equation}
    \label{eq:upper_lower}
    \begin{aligned}
    \Pp(\upperbound(y_n) - \lowerbound(y_n) < \delta_n) & \to 0, \\ 
    \Pp(\upperbound(\pseudoY_n) - \lowerbound(\pseudoY_n) < \delta_n) & \to 0, \\
    \Pp(\min(|\upperbound(y_n)|,|\lowerbound(y_n)|) > 1/\delta_n) & \to 0, \\ 
    \Pp(\min(|\upperbound(\pseudoY_n)|,|\lowerbound(\pseudoY_n)|) > 1/\delta_n) & \to 0. \\
    \end{aligned}
\end{equation}
\end{assumption}

The first two conditions in \eqref{eq:upper_lower} puts a lower bound on the width of
the truncation interval $(\lowerbound(y_n), \upperbound(y_n))$. The last two conditions
makes sure the truncation will not appear too far from the origin and thus we will 
have reasonable behavior in the tail. $\delta_n$ is the rate at which the truncation
interval will shrink (or the distance of the truncation interval to the origin). This 
rate will appear in the RHS of \eqref{eq:general} and thus we impose a condition on 
$(\delta_n, M(\E^*_n, \eta_n), \nrow(\E^*_n), |{\cal S}_n|)$ to ensure the convergence 
of the pivot \eqref{eq:pivot_marginal}. 


\subsection{Main result}

Suppose we have $X_n \in \real^{n \times p_n}$ and $y_n \in \real^n$ is generated according to
\eqref{eq:setup}. We denote its distribution as $\law(y_n|X_n)$. The convergence mentioned
below is under this sequence of distributions $\{\law(y_n|X_n)\}_{n=1}^{\infty}$. 

\begin{theorem}[Convergence of the pivot]
\label{thm:pivot_conv}
Suppose we have a sequence of $y_n$ generated as above with means $\mu_n = \mu(X_n)$, 
and variances $\Sigma_n = \Sigma(X_n)$ and have finite third moments. 
We also assume Assumption \ref{A:denominator} is satisfied with a sequence of $\delta_n$. 
Furthermore, let $\E^*_n$ be a sequence of affine
selection procedures, $\eta_n = \eta(\E^*_n)$, and the corresponding 
$M(\E^*_n, \eta_n)$, $\nrow(\E^*_n)$ and ${\cal S}_n$ properly defined as in 
Section \ref{sec:influence}. Then if  
$$1/\delta_n^6 \cdot M( \E^*_n,\eta_n)^3 \cdot n
\bigg[\log\big(r(\E^*_n)\big)+ \log\big(|{\cal S}_n|\big)\bigg]^4 \to 0, \text{ as }n\to\infty,$$
we have
\begin{equation}
\label{eq:uniform}
P(\eta_n^Ty_n; \eta_n^{T} \Sigma_n \eta_n, \eta_n^T\mu_n, L_{\E^*_n}, U_{\E^*_n}) 
 \overset{d}{\rightarrow}
\unif (0,1), \quad n \rightarrow \infty,
\end{equation}
where $P(x;\sigma^2, m, a,b) = 2 \min(F(x;\sigma^2, m, a,b), 1 - F(x;\sigma^2, m, a,b))$ is
the two-sided pivot. 
\end{theorem}

In the following section, we apply Theorem \ref{thm:pivot_conv} to different selection
procedures.

\section{Examples}
\label{sec:example}

We give two examples in this section as the applications of Theorem
\ref{thm:pivot_conv}. The first example is to perform selective inference
after solving the LASSO and the second is to test the global null in 
generalized linear models. In these two examples, we will explain why
the selection procedure is affine, what is the data distribution 
$\law(y_n|X_n)$ and the quantities  
$(\delta_n, M(\E^*_n, \eta_n), \nrow(\E^*_n), |{\cal S}_n|)$. 
To ease the notations, we suppress the dependencies on $n$ whenever 
possible. It is helpful to keep in mind that $y_n \in \real^n$, $X_n \in \real^{n \times p_n}$.

\subsection{Inference for LASSO with non-Gaussian errors}
Consider the linear model 
\begin{equation}
\label{eq:linear_model}
y = X \beta^0 + \epsilon, \quad X \in \real^{n \times p}, ~~\epsilon_i \overset{iid}{\sim} G(0,\sigma^2),
\end{equation}
where $\sigma$ is known and the distribution $G$ has finite third moments, but is not necessarily Gaussian. 
 
\citet{tibshirani1996regression} proposed the now famous LASSO. 
We get a sparse solution $\hbeta$ by solving
\begin{equation}
\label{eq:Lasso}
\hbeta = \minimize_{\beta \in \real^p} \frac{1}{2}\|y - X \beta\|_2^2 + \lambda \|\beta\|_1, 
\end{equation}
where $\lambda>0$ is the fixed regularization parameter. We choose $\lambda$ as in \cite{negahban_unified_2012}.
If we normalize the columns of $X$ to have norm $1$, \cite{negahban_unified_2012} chooses $\lambda$ to be 
$O(\sqrt{\log p})$.

\subsubsection{Affine selection procedure}
As in \cite{lee2013exact}, we solve \eqref{eq:Lasso} and get a solution $\hbeta$.
Now we consider the selection procedure based on $(E, z_E)$, where 
$$
E = \supp(\hbeta), \qquad z_E = \sign(\hbeta_E), 
$$
where $\hbeta_E$ is $\hbeta$ restricted to the active set $E$. Note this is
different from the selection procedure based only on $E$ but is closely related, 
for detailed discussion see \cite{lee2013exact}. The authors in \cite{lee2013exact}
proved such selection procedure is equivalent to the affine constraints 
$A(E, z_E) y \leq b(E, z_E)$, where 
\begin{equation}
\label{eq:lasso_matrix}
\begin{aligned}
A(E, z_E) &= - \diag(z_E)(X_E^T X_E)^{-1} X_E^T, \\ 
b(E, z_E) &= -\lambda \diag(z_E)(X_E^T X_E)^{-1} z_E. 
\end{aligned}
\end{equation}
To test the hypothesis $H_{0j}: \beta_{j, E} = 0$ for any $j \in E$, we choose 
$\eta$ to be as in \eqref{eq:contrast}.

In this case, a simple calculation will put the number of possible states
at $\cardS = 2^p$, which will cause the bound in \eqref{eq:general} to blow
up when $p>n$. 
However, the choice of $\lambda = O(\sqrt{\log p})$ \citep{negahban_unified_2012}
together with other conditions will ensure $\cardS$ is polynomial in $p$ 
with high probability.

\subsubsection{Number of states $\cardS$ for $\lambda = O(\sqrt{\log p})$}
\label{sec:num_of_states}

Suppose $X$ is column standardized to be mean zero and norm $1$, we first introduce
the {\em restricted strong convexity} condition for matrix $X$.

\begin{definition}[Restricted strong convexity \cite{negahban_unified_2012}]
We say $X\in \real^{n \times p}$ satisfies the restricted strong convexity 
condition for index set $A$ with constant $m > 0$ if 
$$
\|X v\|_2^2 \geq m\|v\|_2^2,
$$
for all $v \in \{\Delta \in \real^p :\|\Delta_{A^c}\|_1 \leq 3\|\Delta_{A}\|_1\}$.
\end{definition}

Now we define the assumptions needed to ensure $\cardS$ is polynomial in
$p$ with high probability.

\begin{assumption}
\label{A:cardS:1}
$X$ satisfies the restriced convexity condition for $A = \supp(\beta^0)$ with constant
$m$, and $\phi_{\text{max}}$, the biggest eigenvalue of $X^T X$ is bounded by a constant $Q$.
\end{assumption}
\begin{assumption}
\label{A:cardS:2}
$\epsilon_i$ are sub-Gaussian errors with known variance $\sigma^2$.
\end{assumption}
\begin{assumption}
\label{A:cardS:3}
The signal is sparse. More specifically, $k = |\supp(\beta^0)|$ is bounded by a constant $K$.
\end{assumption}

Following \cite{negahban_unified_2012}, Lemma \ref{lem:sparsity}
shows with the above assumptions, the effective size of $\cardS$ is polynomial in $p$ with 
high probability. 
\begin{lemma}
\label{lem:sparsity}
With Assumptions \ref{A:cardS:1}-\ref{A:cardS:3}, if we solve \eqref{eq:Lasso} with
$\lambda \geq 4\sigma \sqrt{\log p}$ and get active set $E$, then with probability at least 
$1 - c_1 \exp(-c_1\lambda^2)$, 
$$
|E| \leq \dfrac{16  Q^2}{m^2} \cdot K
$$
where $c_1$ is some constant that depends on $m$ and 
the subgaussian constant of the error $\epsilon$.
Thus, with probability $1 - c_1 \exp(-c_1\lambda^2)$, 
$$
\cardS \leq p^{cK}, \qquad c = \dfrac{16Q^2}{m^2}. 
$$
\end{lemma} 

The proof of Lemma \ref{lem:sparsity} is deferred to the appendix.
Having controlled $\cardS$, now we need to get a bound for the influences. 

\subsubsection{Bounding the influence $M(\E^*,\eta)$}

Assume we have normalized the design matrix $X$ columnwise so that each column has norm $1$.
We further assume the following assumption on $X$,
\begin{assumption}
\label{A:eigenvalue}
Suppose we solve problem \eqref{eq:Lasso} with $X$ and get the active set $E$. Let 
$\phi_{\text{min}}$ be the smallest eigenvalue for submatrices of size less than $n \times |E|$,
more specifically, 
$$
\phi_{\text{min}} = \min_{v \in \real^p, \|v\|_0 \leq |E|} \disp\frac{\|Xv\|^2_2}{\|v\|_2^2}.
$$
We assume $\phi_{\text{min}} \geq \nu > 0$.
\end{assumption}

\begin{lemma}
\label{lem:eigenvalues}
Suppose $X$ satisfies Assumption \ref{A:eigenvalue}, then
$$
\max_{i,j} \bigg|\big((X_E^TX_E)^{-1}X_E^T\big)_{ij}\bigg| \leq \frac{|E|}{\nu^2} \cdot \max_{i,j} |X_{ij}|.
$$
\end{lemma}



\subsubsection{Choice of $\delta_n$ in Assumption \ref{A:denominator}}
If we normalize the columns of $X$ to have norm $1$ and choose $\lambda =
O(\sqrt{\log p})$ in \eqref{eq:Lasso} as in \cite{negahban_unified_2012}.  
Then we assume Assumption \ref{A:denominator} is satisfied with $\delta_n =
O((\sqrt{\log p_n})^{-1-\kappa})$, for any small $\kappa > 0$.

To avoid long passage and stay focused on the main topic, we 
illustrate that Assumption \ref{A:denominator} is satisfied with such $\delta_n$'s
in the following simplified setup. However, the approach can be adapted to 
include more general cases.

\begin{lemma} 
\label{lem:lower_upper}
Suppose Assumption \ref{A:cardS:1}-\ref{A:cardS:3} are satisfied. We further
assume that $z_E = 1$ and the matrix $(X_E^T X_E)^{-1}$ is equicorrelated, i.e.
$$
\begin{aligned}
&\big((X_E^T X_E)^{-1}\big)_{ii} = \big((X_E^T X_E)^{-1}\big)_{jj} = \tau > 0, \\
&\rho = \frac{\big((X_E^T X_E)^{-1}\big)_{ij}}{\big((X_E^T X_E)^{-1}\big)_{ii}} > 0,
\quad \forall i, j \in E, i \neq j.
\end{aligned}
$$ 
Then if 
$\|\beta_n^0\|_{\infty} = O(\lambda_n)$,
Assumption \ref{A:denominator} is satisfied with $\delta_n = O(\lambda_n^{-1-\kappa})$,
for any $\kappa > 0$.
\end{lemma} 

\begin{remark}
Note if we do not assume $z_E = 1$, the last two conditions in Assumption \ref{A:denominator}
are still satisfied with $\delta_n = O(\lambda_n^{-1-\kappa})$ and the first two conditions
can be satisfied with further assumptions. But we do not pursue the technical details here.
\end{remark}

\subsubsection{Convergence of selective tests in the Lasso problems}

Suppose we solve the Lasso problem \eqref{eq:Lasso}
and get active set $E$, and want to test the hypotheses $H_{0j}:\beta_{j,E}  = 0$, we 
can simply take $\eta$ to be as in \eqref{eq:contrast}. 
Now we summarize the above results and apply Theorem \ref{thm:pivot_conv} to get the following corollary 
\begin{lemma}
\label{lem:lasso:conv}
Suppose we solve the Lasso problem \eqref{eq:Lasso} with $\lambda_n = 4\sigma\sqrt{\log p_n}$,
and Assumption \ref{A:denominator}-\ref{A:eigenvalue} are satisfied and the
$\delta_n$'s in Assumption \ref{A:denominator} is chosen as $(\log p_n)^{-\frac{1}{2} - \frac{1}{2}\kappa}$. 
If we further assume
$\max |X_{ij}| = O(n^{-\frac{1}{2}})$, $\|\beta^0\|_{\infty} = O(\sqrt{\log p_n})$, 
and there exists $\kappa > 0$ such that 
$$
n^{-1/2} (\log p_n)^{(7+3\kappa)} \rightarrow 0, 
$$ 
then the pivot in \eqref{eq:uniform} calculated with the appropriate $(\eta_n, L_{\E^*_n}, U_{\E^*_n})$
converges to $\unif(0,1)$. Furthermore, we can construct a test for $\beta_{j,E}$ based on this pivot
that controls ``Type-I error'' \eqref{eq:typeI} asymptotically. 
\end{lemma}

\subsection{Covariance test for $\ell_{1}$-penalized generalized linear models}
\label{sec:glm}

One of the first results in selective inference was the covariance test
\cite{covtest} which provided an asymptotic limiting distribution for the first
step of the Lasso or LARS path. An exact version of this test under Gaussian
errors was described in \cite{taylor2013tests}. 
 
In the following, we generalize the covariance test for generalized linear models.
Suppose $\law(y|x)$ is in an exponential family. More specifically,
\begin{equation} 
\label{eq:exp_family}
p(y|x; \beta^0) = b(y) \exp[(x^{T} \beta^0) y - \Lambda(x^{T} \beta^0)],
\end{equation}
where $\beta^0$ and $x$ are $p$-dimensional vectors
and $\Lambda(\eta)$ is the cumulant generating function of the distribution.

Suppose $y_i|x_i$ are independently distributed according to the law above, where $x_i$'s
are considered fixed. Then
the $\ell_{1}$ penalized generalized linear regression can be expressed as
\begin{equation}
\label{eq:glm}
\hat{\beta}^{*}_{\lambda} = \argmin_{\beta \in \real^p} \sum_{1 \leq i \leq n} 
-\log p(y_i|x_i; \beta) + \lambda \|\beta\|_1. 
\end{equation}

The covariance test for the global null $H_0: \beta^0 = 0$ is based upon the
the first knot on the solution path of \eqref{eq:glm}, which is 
largest score statistic (in absolute values) at $\beta^0 = 0$,
\begin{equation}
\label{eq:first_knot}
\lambda_{1} = \sup \left\{\lambda: \hat{\beta}^*_{\lambda} \neq 0 \right\} = \|X^T(y - \grad\Lambda(0))\|_{\infty}.
\end{equation} 
The variable to achieve the maximum in \eqref{eq:first_knot} will be the first variable to enter the
solution path.

The covariance test can also be viewed as a test for the coefficient with
(potentially) the largest absolute values. A guess for such variable is the first variable to enter
the solution path of \eqref{eq:glm}. In other words, covariance tests select the target of inference 
based on $(j^*, s^*)$, where
\begin{equation}
\label{eq:selection:glm}
(j^*,s^*) = \left(\argmax_{j} |x_j^T(y - \grad\Lambda(0))|, \sign(x_{j^*}^T(y - \grad\Lambda(0)))\right),
\end{equation}
and the test statistic is $\lambda_1 = |x^T_{j^*}(y - \grad\Lambda(0))|$.

\subsubsection{Affine selection procedure}
The selection procedure is based on $(j^*, s^*)$ defined in \eqref{eq:selection:glm}, it is easy to 
see that it is equivalent to 
$$
\begin{aligned}
x_k^T(y-\grad\Lambda(0)) &\leq s^* x^T_{j^*} (y - \grad\Lambda(0)), \quad k=1,\dots,p, \\
-x_k^T(y-\grad\Lambda(0)) &\leq s^* x^T_{j^*} (y - \grad\Lambda(0)), \quad k=1,\dots,p.
\end{aligned}
$$ 
Writing in the form of $A(j^*, s^*) y \leq b(j^*, s^*)$, we have 
$$
A = \begin{pmatrix}
x_1^T - s^*x_{j^*}^T \\
\vdots \\
x_p^T - s^*x_{j^*}^T\\ 
-x_1^T - s^*x_{j^*}^T \\
\vdots \\
-x_p^T - s^*x_{j^*}^T 
\end{pmatrix}, \quad
b = \begin{pmatrix}
(x_1^T - s^*x_{j^*}^T) \grad\Lambda(0)\\
\vdots \\
(x_p^T - s^*x_{j^*}^T) \grad\Lambda(0)\\
-(x_1^T + s^*x_{j^*}^T) \grad\Lambda(0)\\
\vdots \\
-(x_p^T + s^*x_{j^*}^T) \grad\Lambda(0)
\end{pmatrix}.
$$
We notice that $\lambda_1 = s^* x^T_{j^*}(y - \grad\Lambda(0))$. Thus 
to test the global $H_0: \beta^0 = 0$, we simply take 
$$
\eta = s^* x_{j^*}.
$$

The challenge in establishing a result for the covariance test for GLM
is the lack of Gaussianity in the data distribution. The tools we develop
in this paper, however, can circumvent this. But we first need to establish
the resulting pivot which we can use to test the hypothesis $H_0: \beta^0 = 0$. 
Note that $y_i\mid x_i \overset{\text{ind}}{\sim} G(\mu(x_i), \sigma(x_i)^2)$,
thus if $G$ were normal distribution, we will have an exact pivot by applying
Theorem \ref{thm:general_gaussian}. This result is also given in \citet{taylor2013tests}.
Formally, we have the following corollary.


\begin{corollary}[Global test for Gaussian errors]
\label{cor:gaussian_test}
Suppose $\pseudoY_i \overset{ind}{\sim} N(\mu_i,\sigma_i^2)$, $X \in \real^{n
\times p}$ fixed, define $\mu = (\mu_1, \dots, \mu_n)$, $\Sigma = diag\{\sigma_1^2,
\dots, \sigma_n^2\}$. After getting the first knot on the solution path of
\eqref{eq:glm}, we get $(j^*, s^*)$ as defined in \eqref{eq:selection:glm} and
$\lambda_1 = |x_{j^*}^T(\pseudoY - \mu)|$. Furthermore, we also define $\Theta_{jk} = x_j^T\Sigma
x_k$ and  
\begin{align*}
\smash{\lowerboundj} &= \sup_{\substack{(s,k): s \in \{-1,1\}, k \neq j^*,\\ 1-s s^* 
\Theta_{j^* k}/ \Theta_{j^* j^*} > 0}}  \frac{s(x_k - \Theta_{j^* k} 
/ \Theta_{j^* j^*} x_{j^*})^{T} (\pseudoY - \mu)} {1 - s s^{*} \Theta_{j^* k} / \Theta_{j^* j^*}},\\ 
\smash{\upperboundj} &= \inf_{\substack{(s,k): s \in \{-1,1\}, k \neq j^*,\\ 1-s s^* 
\Theta_{j^* k}/ \Theta_{j^* j^*} < 0}}  \frac{s(x_k - \Theta_{j^* k} 
/ \Theta_{j^* j^*} x_{j^*})^{T} (\pseudoY - \mu)} {1 - s s^{*} \Theta_{j^* k} / \Theta_{j^* j^*}}, 
\end{align*} 

Then,
\begin{equation}
\label{eq:gaussian_test}
\disp\frac{\Phi \left(\frac{\smash{\upperboundj}} {\Theta_{j^* j^*}}\right) 
- \Phi\left(\frac{\lambda_1} {\Theta_{j^* j^*}}\right)} {\Phi 
\left(\frac{\smash{\upperboundj}} {\Theta_{j^* j^*}}\right) 
- \Phi\left(\frac{\smash{\lowerboundj}} {\Theta_{j^* j^*}}\right)}
\sim \unif (0, 1) 
\end{equation} 
\end{corollary}

Corollary \ref{cor:gaussian_test} gives a pivot \eqref{eq:gaussian_test} 
which we can use to test the global null $H_0: \beta^0 = 0$ and control the
``Type-I error'' \eqref{eq:typeI}.
In practice, we often normalized the columns of the design matrix $X$. 
In addition we may assume the observations $y_i$'s are independently
distributed with the same marginal variance, i.e. $\Sigma = \sigma^2\text{I}$, then
$U_{(j^*, s^*)} = \infty$ and
$L_{(j^*, s^*)}$ simplifies to the second knot in the solution path $\lambda_2$,
thus we have:
\begin{equation}
\label{eq:second_knot}
\disp\frac{1 - \Phi(\lambda_1)}{1 - \Phi(\lambda_2)} \sim \unif(0, 1).
\end{equation}

For the pivot \eqref{eq:second_knot} to converge to $\unif(0,1)$, 
we need to consider the number of states $\cardS$, the bound on the influence $M(j^*, s^*)$
as well as the choice of $\delta_n$ in Assumption \ref{A:denominator}.

\subsubsection{The conditions for the pivot to converge}
Since $j^* \in \{1, \dots, p\}$, and $s^* \in \{-1, 1\}$, the number of possible states
$\cardS$ are naturally bounded by $2p$ and $r(\E^*) = 2p$. We assume  
$\Sigma = \sigma^2 \text{I}$, and $X$ are normalized columnwise to have norm $1$.
We first introduce the following condition on the design matrix $X$,
which states that any two columns of $X$ cannot be too correlated. 

\begin{assumption}
\label{A:correlation}
Suppose there exists $\rho > 0$, such that 
$$
x_i^T x_j \leq \rho^2 < 1, \qquad \forall~ i \neq j,~ i,j \in \{1,2, \dots, p\}.
$$
\end{assumption}

Under Assumption \ref{A:correlation}, it is not hard to verify 
$$
M(j^*, s^*) \leq \frac{2\max_{ij} |X_{ij}|}{1-\rho^2}.
$$
Now we need to pick the $\delta_n$'s such that Assumption \ref{A:denominator}
holds. In particular, we choose $\delta_n = (\sqrt{\log p_n})^{-1-\kappa}$, for some
$\kappa > 0$. Now if we apply Theorem \ref{thm:pivot_conv}, we have the following result,

\begin{corollary}
\label{cor:nongaussian_test}
Suppose $y \vert X$ is generated independently coordinate-wise through the distribution in 
\eqref{eq:exp_family} with the same marginal variance. Assume the columns of 
$X$ have norm $1$, Assumption \ref{A:correlation} is satisfied and 
$\max_{ij} |X_{ij}| = O(n^{-\frac{1}{2}})$. Then if 
$$
\big(\log p_n\big)^{7+3\kappa} n^{-\frac{1}{2}} \rightarrow 0,
$$
the pivot converges to $\unif(0,1)$ under the global null $H_0: \beta^0 = 0$,
$$
\disp\frac{1 - \Phi(\lambda_1)}{1 - \Phi(\lambda_2)} \overset{d}{\rightarrow} \unif(0, 1).
$$
\end{corollary}

\section{Proof of the theorems}
\label{sec:proofs}

Without loss of generality, we restrict our interest to the case $\mu=\mu(X)=0, \Sigma= \Sigma(X)=I$. This is possible
since any affine selection procedure $\E^*$ applied to data with mean $\mu(X) \neq 0$ is equivalent
to a centered affine selection procedure $\E^{*,0}$ applied to the centered data.
Specifically, the linear part of $\E^{*,0}$ is the same as $\E^*$ and the offsets are related by
$$
b^{0}(\E) = b(\E) - A(\E)\mu.
$$
Further, note that all quantities in the theorems above are independent of $b$. Scaling of the errors is handled
in a similar fashion.

\subsection{Proof of Theorem \ref{thm:general_gaussian}}

\label{sec:proof}
Analogous to the proof in \citet{lee2013exact},
we prove Theorem \ref{thm:general_gaussian}.
\begin{proof}
To lighten notations, we suppress all dependencies on $X$ as it is assumed known.
Note that $\{\E^*(\pseudoY) = \E\} = \{A(\E)  \pseudoY \leq b(\E)\}$.  
Thus
$$
\law\bigg(\eta(\E)^T  \pseudoY | \E^*(\pseudoY) = \E\bigg) \overset{d}{=} \law\bigg(\eta(\E)^T \pseudoY |
A(\E)  \pseudoY \leq b(\E)\bigg).
$$
Dropping the dependence on $\E$ for the moment,
\begin{align*}
\{A  \pseudoY \leq b\} &= \{A  \pseudoY - \Ee[A  \pseudoY| \eta^T  
\pseudoY] \leq b - \Ee[A  \pseudoY | \eta^T  \pseudoY]\}, \\
& = \{\Ee[A  \pseudoY | \eta^T  \pseudoY] 
\leq b - (A  \pseudoY - \Ee[A  \pseudoY | \eta^T  \pseudoY])\} \\
& = \{\alpha \eta^T  \pseudoY \leq b - A  \pseudoY + \alpha \eta^T  \pseudoY\} \\ 
& = \{\alpha_j \eta^T  \pseudoY \leq b_j - (A  \pseudoY)_j + \alpha_j 
\eta^T  \pseudoY, ~ j = 1, \dots, k\}.\\
\end{align*}
In other words,
$
\{A \pseudoY \leq b \} = \{A (\E) \pseudoY \leq b(\E)\} 
= \{\lowerboundfix( \pseudoY) \leq \eta^T  \pseudoY \leq \upperboundfix( \pseudoY)\},
$
and
$$
\law\bigg(\eta(\E)^T \pseudoY|\E^*(\pseudoY) = \E\bigg) \overset{d}{=}  \law\bigg(\eta(\E)^T  \pseudoY  \, |  \,
\lowerboundfix( \pseudoY) \leq \eta(\E)^T  \pseudoY \leq 
\upperboundfix( \pseudoY)\bigg),
$$
Note also from the derivation above that $(\lowerboundfix(\pseudoY),\upperboundfix(\pseudoY))$ is independent of $\eta^T  \pseudoY$ for
each $\E$.
Thus if we condition on $\E^*$, $\upperbound(\pseudoY)$ and $\lowerbound(\pseudoY)$, $\eta(\E^*)^T  \pseudoY$ is 
distributed as a Gaussian r.v. with mean $0$ and variance $\|\eta(\E^*)\|^2$
truncated at $\upperbound$ and $\lowerbound$. Therefore,
$$
F(\eta^T\pseudoY; \|\eta\|^2, 0, \lowerbound(\pseudoY) , \upperbound(  \pseudoY))
| \E^* = \E, \lowerbound(\pseudoY), \upperbound(\pseudoY) \sim \unif(0,1).
$$
Considering that conditional on $\E^*$, 
$\eta(\E^*)^T \pseudoY$ is independent of $\upperbound$ and $\lowerbound$, 
we have \eqref{eq:pivot}.
\end{proof}

\subsection{Smoothing the maxima of affine functions}
\label{sec:smoothing}

In the proof of Theorem \ref{thm:general} and the related lemmas and corollaries, a technique developed
by \cite{chatterjee2005simple}
is frequently used. Roughly speaking, we want to study convergence of functions like 
$\lowerboundfix$ and $\upperboundfix$ which can be expressed
as maxima or minima of affine functions. These non-smooth functions are replaced by a smoothed
surrogate at the cost of a factor appearing in their derivatives depending on the smoothing parameter.

Specifically, we are interested in how this smoothing affects the following quantities.  
\begin{definition}
For any  $f \in C^3(\real^n; \real^q)$, define
\begin{equation}
\label{eq:semiorm}
\lambda_r(f) = \sup\left\{\bigg[|\partial_l^p f_k (x)|\bigg]^{r/p}, 1 \leq p \leq r,
l = 1, \dots,n, k = 1, \dots, q, x \in \real^n \right\}.
\end{equation}
For any finite collection ${\cal F}$ of functions define
$$
\lambda_r({\cal F}) = \max_{f \in {\cal F}} \lambda_r(f).
$$
\end{definition}

\begin{definition}
    For any $g \in C^3(\mD, \real)$ where $\mD \subseteq \real^3$ and any multi-index $\alpha = (\alpha_1, \alpha_2, \alpha_3)$,
    we define for $r=1, 2,3$,
    $$
    C_r(g) = \max\left[ \sup\left\{\bigg|\dfrac{\partial^{|\alpha|}g(u,v,w)}{\partial u^{\alpha_1} \partial v^{\alpha_2} 
        \partial w^{\alpha_3}} \bigg|^{r/|\alpha|}, (u,v,w) \in \mD, |\alpha| \leq r\right\}, 1\right].
    $$
\end{definition}

Now we define the smoothed maxima operator.

\begin{definition}
    Let $f: \real^n \rightarrow \real^3$ and define $f = \max_{v \in \mathcal{F}} v$, where 
    $$
    \mathcal{F} = \left\{z \mapsto v_j(z), v_j \in C^3(\real^n, \real^3) \right\} 
    $$ 
    is a finite collection of thrice differentiable functions $v_j$'s. The maximum is taken
    coordinate-wise.

    We define the smoothed maxima operator with parameter $\beta$ as
    \begin{equation}
    \Gamma(f, \beta) = \frac{1}{\beta} \log \left[\sum_{v_j \in \mathcal{F}} 
    \exp \left( \beta v_j \right) \right] \in C^3(\real^n, \real^3), 
    \end{equation}
    where the operators $\log$ and $\exp$ are applied coordinate-wise.
\end{definition}

Suppose the range of $f, \Gamma(f, \beta)$, denoted as $\mathcal{R}(f), \mathcal{R}(\Gamma(f,\beta)) \subseteq \mD$
and let $h = g \circ f$, $h_{\beta} = g \circ \Gamma(f, \beta)$, then Lemma \ref{lem:smooth_max} gives a bound on 
$\|h - h_{\beta}\|_{\infty}$ and $\lambda_3(h_{\beta})$.

\begin{lemma}
\label{lem:smooth_max}
Assume the same notations as above, $s = |\mathcal{F}|$, then for $\beta \geq 1$ 
\begin{align}
&   \|h - h_{\beta}\|_{\infty} \leq C_1(g) \cdot \frac{3}{\beta} \log s, \label{eq:diff} \\
&   \lambda_3(h_{\beta}) \leq 13c \cdot \beta^2 C_3(g) \lambda_3(\mathcal{F}), \label{eq:deriv}
\end{align}
where $c$ is a universal constant.
\end{lemma}

The proof of Lemma \ref{lem:smooth_max} will refer to the following lemma whose proof we leave in the Appendix.
\begin{lemma}
\label{lem:derivative}
For any $f \in C^3(\real^n;\real^3)$ and $g \in C^3(\real^3, \real)$, $r=1,2,3$ 
\begin{equation}
\label{eq:derivative1}
\lambda_r(g \circ f) \leq cC_r(g) \cdot \lambda_r(f), \qquad \forall l=1,2,\dots,n, ~ x \in \real^n,
\end{equation}
where $c$ is a universal constant.
\end{lemma}

Now we prove Lemma \ref{lem:smooth_max}.
\begin{proof}
Note that for any $u \in \real^s$
\begin{align*}
\max_{1 \leq j \leq s} u_j &= \frac{1}{\beta} \log \left[\exp \left(\beta \max_{1 \leq j \leq s}
u_j\right) \right] \\
& \leq \frac{1}{\beta}  \log \left[\sum_{j=1}^s
\exp\left(\beta u_j\right) \right] \\
& \leq \frac{1}{\beta} \log \left[s \exp\left(\beta \max_{1 \leq j \leq s}
u_j)\right)\right] \\
& = \frac{1}{\beta} \log s + \max_{1 \leq j \leq s} u_j.
\end{align*}
We take $h = g\circ f$, and $h_{\beta} = g \circ \Gamma(f,\beta)$,
\begin{align*}
|h(z) - h_{\beta}(z)| &= |g \circ f(z) - g \circ \Gamma(f, \beta)(z)| \\ 
&\leq 3C_1(g)  \|f(z) - \Gamma(f, \beta)(z)\|_{\infty} \\
&\leq 3C_1(g) \cdot \frac{1}{\beta} \log s, 
\end{align*}
where the $\infty$ norm is the element-wise maximum absolute value. Thus we proved \eqref{eq:diff}. 
Now let $f=(f_1, f_2, f_3)$ and $v_j =(v_{1j}, v_{2j}, v_{3j})$, and define
$$
\cF_i = \left\{z \mapsto v_{ij}(z), v_{ij} \in C^3(\real^n, \real) \right\}, \qquad i=1,2,3. 
$$
Theorem 1.3 in \citet{chatterjee2005simple} proved that
\begin{equation}
\lambda_3(\Gamma(f_i, \beta)) \leq 13 \beta^2 \lambda_3(\mathcal{F}_i), \qquad \forall \beta \geq 1,
~i=1,2,3. 
\end{equation}
Note $\lambda_3\big(\Gamma(f,\beta)\big) = \max_{i=1,2,3}\lambda_3\big(\Gamma(f_i,\beta)\big)$, and
that $\lambda_3(\cF) = \max_{i=1,2,3} \lambda_3(\cF_i)$, thus 
\begin{equation}
\lambda_3(\Gamma(f, \beta)) \leq 13 \beta^2 \lambda_3(\mathcal{F}). 
\end{equation}
This combined with Lemma \ref{lem:derivative} proves \eqref{eq:deriv}.
\end{proof}

\subsection{Proof of Theorem \ref{thm:general}}

To prove Theorem \ref{thm:general}, we first prove the following
lemma. Recall our reduction to the standard Gaussian $N(0,I)$ in the beginning of
Section \ref{sec:proof}. Lemma \ref{lem:general} is a simple adaption of Lindberg's proof of the CLT.
\begin{lemma}
\label{lem:general}
Assuming the same notation as in Theorem \ref{thm:general}, for any smooth function $h \in C^3(\real^n)$, 
\begin{equation}
\label{eq:lindberg}
|\Ee h(y) - \Ee h(\pseudoY)| \leq  \frac{1}{6} \lambda_3(h) n \max_l(\gamma, \Ee(|\pseudoY_l|^3)) 
\end{equation}
\end{lemma}

We will prove Lemma \ref{lem:general} now.
\begin{proof}
%
%
%
The proof proceeds by following the Lindberg proof of the CLT for $h$. Define
\begin{align*}
y^l &= (y_1, y_2, \dots, y_l, \pseudoY_{l+1}, \dots, \pseudoY_n), \\
\pseudoY^l &= (y_1, y_2, \dots, y_{l-1}, \pseudoY_{l}, \dots, \pseudoY_n), \\
W^l &= (y_1, y_2, \dots, y_{l-1}, 0, \pseudoY_{l+1}, \dots, \pseudoY_n).
\end{align*}
We can break the absolute difference of the two expectations into $n$ parts,
\begin{equation}
|\Ee h(y) - \Ee h(\pseudoY)| \leq \sum_{l=1}^{n} |\Ee h(y^l) - \Ee h(\pseudoY^l)|.
\end{equation}
Note that
\begin{align*}
h(y^l) - h(W^l) &= \partial_l h(W^l) y_l + \frac{1}{2} \partial_l^2 h(W^l) y_l^2
+ R^l,\\
h(\pseudoY^l) - h(W^l) &= \partial_l h(W^l) \pseudoY_l + \frac{1}{2} \partial_l^2 
h(W^l) \pseudoY_l^2 + T^l,\\
\end{align*}
where $|R^l| \leq \frac{1}{6} \|\partial_l^3 h\|_{\infty} |y_l|^3$,
$|T^l| \leq \frac{1}{6} \|\partial_l^3 h\|_{\infty} |\pseudoY_l|^3$.
Moreover, because $y_l$'s and $\pseudoY_l$'s are independent,
$W^l$ is independent of both $y_l$ and $\pseudoY_l$.
Continuing, we see the first and second order differences cancel out,
\begin{align*}
&|\Ee h(y^l) - \Ee h(\pseudoY^l)|\\
\leq& |\Ee \partial_l h(W^l)(y_l -\pseudoY_l)|
+ \left|\frac{1}{2} \Ee \partial_l^2 h(W^l)(y_l^2 -\pseudoY_l^2) \right| + \Ee|R^l - T^l|\\
=& |\Ee \partial_l h(W^l) \Ee(y_l -\pseudoY_l)|
+ \left|\frac{1}{2} \Ee \partial_l^2 h(W^l) \Ee(y_l^2 -\pseudoY_l^2) \right| +\Ee|R^l - T^l| \\
 =& \Ee|R^l - T^l|.
\end{align*}
Combining the $n$ parts, we have
\begin{align*}
    |\Ee h(y) - \Ee h(\pseudoY)| & \leq  \frac{1}{6} \lambda_3(h) \sum_{l=1}^n \left[
    \Ee(|\pseudoY_l|^3) + \Ee(|y_l|^3)\right] \\ 
    & \leq \frac{1}{6} \lambda_3(h) n \max_l(\gamma, \Ee(|\pseudoY_l|^3)) 
\end{align*}
\end{proof}

Now we turn to the proof of Theorem \ref{thm:general}.
\begin{proof}
Note that since 
$\{A(\E_i,X) y \leq b(\E_i)\}, 1 \leq i \leq \cardS$ are disjoint and for any state $\E$
$$
\{A(\E,X) y \leq b(\E)\} = \{\lowerboundfix(y)\leq \eta(\E)^T y \leq \upperboundfix(y)\}.
$$

Therefore the quantity of interest is
$$
W(\eta(\E^*)^T y, \lowerbound(y), \upperbound(y))
= \max_{1 \leq i \leq  \cardS} W(\eta(\E_i)^T y, \loweri(y), \upperi(y))
$$

If we knew the above quantity was smooth with respect to the data, we can apply Lemma \ref{lem:general} directly.
However, there are two non smooth expressions above: the maximum function over the states and
in $\loweri$ and $\upperi$. We smooth each and optimize over the smoothing parameter.

We first smooth over $\upperboundfix$ and $\lowerboundfix$. Note that $\forall z \in \real^n$
\begin{align*}
\lowerboundfix(z) &= \max_{\alpha_j < 0} \frac{b(\E)_j - (A(\E) z)_j + 
\alpha_j \eta^T z}{\alpha_j} \\
&= \max \{v | v \in \mathcal{F}_L\},\\
\upperboundfix(z) &= \min_{\alpha_j > 0} \frac{b(\E)_j - (A(\E) z)_j + 
\alpha_j \eta^T z}{\alpha_j} \\
&= \min \{v | v \in \mathcal{F}_U\},
\end{align*}
where $\mathcal{F}_L, \mathcal{F}_U$ are the collections of affine functions
$$
\begin{aligned}
\mathcal{F}_L &= \left\{v_j: z \mapsto \left(\frac{b(\E)_j - (A(\E) z)_j + 
\alpha_j \eta^T z}{\alpha_j} \right) \bigg|  \alpha_j < 0 \right\} \\
\mathcal{F}_U &= \left\{v_j: z \mapsto \left(\frac{b(\E)_j - (A(\E) z)_j + 
\alpha_j \eta^T z}{\alpha_j} \right) \bigg|  \alpha_j > 0 \right\}.
\end{aligned}
$$
Finally, note that
\begin{equation}
    \label{eq:bounded_influence}
\max(\lambda_3(\mathcal{F}_L), \lambda_3(\mathcal{F}_U)) \leq M(\E,\eta)^3.
\end{equation}

We define the smoothing parameter $\beta = 1 / \delta$, for some $0<\delta<1$. Then
\begin{equation}
    \label{eq:upper_lower_approx}
    \begin{aligned}
        \tlowerdelta (z) &= \Gamma(\lowerboundfix(z), 1/\delta) = \delta \log \left[\sum_{v_j \in \mathcal{F}_L} 
        \exp \left( \dfrac{v_j(z)}{\delta} \right)  \right], \\ 
        \tupperdelta (z) &= -\Gamma(-\upperboundfix(z), 1/\delta) = -\delta \log \left[\sum_{v_j \in \mathcal{F}_U} 
        \exp \left(- \dfrac{v_j(z)}{\delta} \right) \right],
    \end{aligned} 
\end{equation}

By Lemma \ref{lem:smooth_max}, for any selection state $\E$ and $z \in \real^n$, we see
\begin{equation}
\label{eq:first_approx}
\begin{aligned}
& |\lowerboundfix(z) - \tlowerdelta(z)| \leq \delta \log r(\E^*) \\
& |\upperboundfix(z) - \tupperdelta(z)| \leq \delta \log r(\E^*) 
\end{aligned}
\end{equation}

Based on \eqref{eq:first_approx} and Lemma \ref{lem:smooth_max}, we have
\begin{equation}
    \begin{aligned}
        & |\Ee W(\eta(\E^*)^T y, \lowerbound(y), \upperbound(y)) -
        \Ee W(\eta(\E^*)^T \pseudoY, \lowerbound(\pseudoY), \upperbound(\pseudoY))| \\
        \leq & |\Ee W(\eta(\E^*)^T y, \tlowerstar(y), \tupperstar(y)) - 
        \Ee W(\eta(\E^*)^T \pseudoY, \tlowerstar(\pseudoY), \tupperstar(\pseudoY))| \\
        & \hspace{10pt}+ 12 \delta C_1(W) \log r(\E^*) \\
        = & |\Ee \max_{\E_i} W(\eta(\E_i)^T y, \tloweridelta(y), \tupperidelta(y)) - 
        \Ee \max_{\E_i} W(\eta(\E_i)^T \pseudoY, \tloweridelta(\pseudoY), \tupperidelta(\pseudoY))| \\
        & \hspace{10pt}+ 12 \delta C_1(W) \log r(\E^*) 
    \end{aligned}
\end{equation}
The equality follows from the fact that 
$\upperboundfix \geq \tupperdelta$ and $\lowerboundfix \leq \tlowerdelta$, for any state $\E$ and that
$W$ is supported on $D = \{(u,v,w) | v \leq u \leq w\}$. Therefore, 
$W(\eta(\E^*)^T y, \tlowerstar(y), \tupperstar(y)) = \max_{\E_i} W(\eta(\E_i)^T y, \tloweridelta(y), \tupperidelta(y))$.

Next, we smooth the maximum over states. Define
$$
W_{\E}(z) = W(\eta^T z, \tlowerdelta(z), \tupperdelta(z)) \qquad \E \in \mathcal{S}, ~z \in \real^n.
$$
We also define the smoothed maxima for $\max_{1 \leq i \leq |S|} W_{\E_i}(z)$,
$$
H_{\delta}(z) = \Gamma(\max_{1 \leq i \leq |S|} W_{\E_i}(z), 1/\delta)
= \delta \log \left[\sum_i \exp \left(\frac{W_{\E_i}(z)}{\delta} \right)\right].
$$

Thus by Lemma \ref{lem:smooth_max}, 
\begin{equation}
    \label{eq:smoothing}
    \begin{aligned}
        & |\Ee W(\eta(\E^*)^T y, \lowerbound(y), \upperbound(y)) -
        \Ee W(\eta(\E^*)^T \pseudoY, \lowerbound(\pseudoY), \upperbound(\pseudoY))| \\
        \leq & 12 \delta \log \cardS + 12 \delta C_1(W) \log r(\E^*)
        + |\Ee H_{\delta}(y) - \Ee H_{\delta}(\pseudoY)|,
    \end{aligned}
\end{equation}

For any state $\E$, define $\tf_{\delta, \E}: z \mapsto (\eta(\E)^T z, \tlowerdelta(z), \tupperdelta(z))$. 
Theorem 1.3 in \cite{chatterjee2005simple} states that
\begin{equation}
    \label{eq:lambda3}
    \lambda_3(H_{\delta}) \leq \frac{13}{\delta^2} \max_i \lambda_3(W \circ \tf_{\delta, \E_i}).
\end{equation}

Per Lemma \ref{lem:smooth_max} and inequality \eqref{eq:bounded_influence} 
\begin{equation}
    \label{eq:lam3}
    \lambda_3(W \circ \tf_{\delta, \E}) \leq \frac{13c}{\delta^2} C_3(W) M(\E^*, \eta)^3.
\end{equation}

From \eqref{eq:lindberg} in Lemma \ref{lem:general} together with \eqref{eq:lambda3} and
\eqref{eq:lam3}, we have
\begin{equation}
    |\Ee H_{\delta}(y) -\Ee H_{\delta}(\pseudoY)|
    \leq \frac{169c}{6\delta^4} C_3(W) n M(\E^*,\eta)^3 
    \max(\gamma, \Ee[\pseudoY_l^3]) 
\end{equation}

All the above combined, we have
\begin{align*}
    & |\Ee W(\eta(\E^*)^T y, \lowerbound(y), \upperbound(y)) -
    \Ee W(\eta(\E^*)^T \pseudoY, \lowerbound(\pseudoY), \upperbound(\pseudoY))| \\
    \leq & 12 \delta \left[\log \cardS + C_1(W) \log r(\E^*)\right]
    + \frac{169c}{6} \frac{1}{\delta^4} \cdot C_3(W) n M(\E^*, \eta)^3 
    \max(\gamma, \Ee[\pseudoY_l^3]).
\end{align*}
Notice that for the last inequality to hold, we require $\delta \leq 1$. 
But the optimal $\delta^5 = O(nM(\E^*, \eta)^3 / [\log \cardS + \log r(\E^*)])$, which will go to $0$ since the numerator 
shrinks to $0$ while the denominator goes to $\infty$ as $n,p \to \infty$. 
Therefore, the inequality holds.

Optimizing over $\delta$ yields \eqref{eq:general}.
\end{proof}

\subsection{Proof of Theorem \ref{thm:pivot_conv}}
Now let's turn to the proof of our main result, Theorem \ref{thm:pivot_conv},
\begin{proof}
    For the convenience of notation, we denote $P(x;\sigma^2, m, a, b)$ by
    $P(x; a, b)$, omitting $\sigma^2, m$ in the following proof. 
    Define
    $$
    D(\delta) = \left\{(x,a,b): a \leq x \leq b, b-a \geq \delta, \min(|b|,|a|) \leq 1/\delta \right\}.
    $$
    We claim that for any small $\frac{1}{4} > \delta > 0$, we can find a 
    thrice differentiable function $\tP_{\delta}$ such that
    \begin{align*}
        & \tP_{\delta} \text{ is supported on the set } \{a \leq x \leq b\} \\
        &\norm{\tP_{\delta} - P}_{\infty} \leq (K_1 + 1)\delta \text{ on the set } D(\delta), \\ 
        &C_3(\tP_{\delta}) \leq K_3 \cdot \frac{1}{\delta^6} \text{ on the set } D(\delta),
    \end{align*}
    where $K_1$ and $K_3$ are defined in Corollary \ref{cor:derivatives}.

    The proof of the existence of such a function $\tP_{\delta}$ is left to
    Lemma \ref{lem:existence} in the Appendix.
Then for any positive, bounded
function $\Psi \in C^3( \real^+; \real^+), ~ \Psi(0)=0$, with bounded third derivatives, we have:
    \begin{equation}
        \label{eq:prob}
    \begin{aligned}
        &|\Ee \Psi \circ \tP_{\delta}(y; \lowerbound, \upperbound) - \Ee \Psi \circ P(y; \lowerbound, \upperbound)| \\
        \leq & \norm{\Psi'}_{\infty}(K_1 + 1)\delta + 2\norm{\Psi}_{\infty}\cdot \Pp\bigg(\upperbound(y) - \lowerbound(y) < \delta \bigg)\\
         +& 2\norm{\Psi}_{\infty} \cdot \Pp\bigg(\min(|\upperbound(y)|, |\lowerbound(y)|) > 1/\delta\bigg),\\
        &|\Ee \Psi \circ \tP_{\delta}(\pseudoY; \lowerbound, \upperbound) - \Ee \Psi\circ P(\pseudoY; \lowerbound, \upperbound)| \\
        \leq & \norm{\Psi'}_{\infty}(K_1 + 1)\delta + 2\norm{\Psi}_{\infty}\Pp\bigg(\upperbound(\pseudoY) - \lowerbound(\pseudoY) < \delta \bigg)\\
         +& 2\norm{\Psi}_{\infty}\Pp\bigg(\min(|\upperbound(\pseudoY)|, |\lowerbound(\pseudoY)|) > 1/\delta\bigg),
    \end{aligned}
    \end{equation}

    On the other hand, we plug in $\Psi \circ \tP_{\delta}$ as the $W$ in Theorem \ref{thm:general},
    then for any sequence of $\tP_{\delta_n}$,
    \begin{equation}
        \label{eq:subseq}
        |\Ee \Psi\circ \tP_{\delta_n}(y) - \Ee \Psi\circ \tP_{\delta_n}(\pseudoY)| \leq C 
        \bigg[K_3 \cdot \dfrac{1}{\delta_n^6}(\log r(\E^*) \cardS)^4 n M(\E^*, \eta)^3\bigg]^{\frac{1}{5}}. 
    \end{equation}
    If we choose a subsequence $\delta_n \rightarrow 0$, such that the right hand side of \eqref{eq:subseq} goes to
    zero, then
    \begin{equation}
        \label{eq:corproof}
        \begin{aligned}
            &|\Ee \Psi\circ P(y) - \Ee \Psi\circ P(\pseudoY)| \\
            \leq & |\Ee \Psi\circ \tP_{\delta_n}(y) - \Ee \Psi\circ \tP_{\delta_n}(\pseudoY)|+ 2\norm{\Psi'}_{\infty}(K_1 + 1)\delta_n \\ 
            +& 2\norm{\Psi}_{\infty}\Pp\bigg(\upperbound(y) - \lowerbound(y) < \delta_n\bigg) \\ 
             +& 2\norm{\Psi}_{\infty}\Pp\bigg(\min(|\upperbound(y)|, |\lowerbound(y))| > 1/\delta_n\bigg)\\ 
            +& 2\norm{\Psi}_{\infty}\Pp\bigg(\upperbound(\pseudoY) - \lowerbound(\pseudoY) < \delta_n\bigg) \\
             +& 2\norm{\Psi}_{\infty}\Pp\bigg(\min(|\upperbound(\pseudoY)|, |\lowerbound(\pseudoY)|) > 1/\delta_n\bigg)\\
            & \rightarrow 0.
        \end{aligned}
    \end{equation}
    Note that the right hand side of \eqref{eq:corproof} goes to $0$ because of Assumption \ref{A:denominator}.
Thus we have the conclusion that $P(y) \overset{d}{\rightarrow} P(\pseudoY) \sim \unif(0,1)$. 
    
\end{proof}

\section{Discussion}
\label{sec:discussion}
This work proves a generic framework in which asymptotic results hold for many selective inference problems. 
It is, however, not directly applicable to some other procedures. Further work may include,
\begin{enumerate}
    \item Fixed $\lambda$ for generalized linear model.

    Our work derives a theory for inference after the affine selection procedure. However, 
    inference for a fixed $\lambda$ for the generalized linear regression is 
    not an affine selection procedure. A plausible solution will be to approximate the loss function of GLM by
    a quadratic form and bound the difference between the quadratic form and the GLM loss function. However, this
    is still an open question.

    \item Apply the result to nonparametric problems.

    It is a big step to remove the Gaussian assumptions required by \cite{lee2013exact} which restricts our attention
    to Gaussian families. Without the Gaussian constraints, we can consider some exponential families and potentially
    some nonparametric problems as well. 
\end{enumerate}

{\bf Acknowledgements} We thank Jason Lee for the proof of Lemma \ref{lem:sparsity} provides a polynomial bounds on the number of 
selection states. We also thank Will Fithian and Rob Tibshirani for discussions on potential applications of our result. 
We also thank the anonymous reviewer who has read our paper so carefully and provide constructive advice for reorganization.

\bibliographystyle{agsm}
\bibliography{nongaussian}

\appendix
\section{Proof of Lemma \ref{lem:derivative}}
\begin{proof}
    Using the chain rules, we have the second derivatives with respect to $x_l, ~l=1,2,\dots,n$ as 
    \begin{equation}
        \frac{\partial^2 g\circ f}{\partial x_l^2} = \sum_{i,j=1}^3\frac{\partial^2 g}{\partial f_i \partial f_j} 
        \left(\frac{\partial f_i}{\partial x_l} \frac{\partial f_j}{\partial x_l}\right) 
        +\sum_{i=1}^3\frac{\partial g}{\partial f_i} \frac{\partial^2 f_i}{\partial x_l^2}
    \end{equation}
    and the third derivatives with respect to $x_l, ~l=1,2,\dots,n$ as 
    \begin{equation}
        \begin{aligned}
            \frac{\partial^3 g\circ f}{\partial x_l^3} &= \sum_{i,j,k=1}^3 \frac{\partial^3 g}{\partial f_i \partial f_j \partial f_k} 
            \left(\frac{\partial f_i}{\partial x_l} \frac{\partial f_j}{\partial x_l} \frac{\partial f_k}{\partial x_l}\right) 
            + 3\sum_{i,j=1}^3 \frac{\partial^2 g}{\partial f_i \partial f_j} \left(\frac{\partial^2 f_i}{\partial x_l^2} \frac{\partial f_j}{\partial x_l}\right)\\
            &+ \sum_{i=1}^3 \frac{\partial g}{\partial f_i} \frac{\partial^3 f_i}{\partial x_l^3}.
        \end{aligned}
    \end{equation}
    
    For $r=1$, the conclusion is obviously true with the constant $c=3$. For $r=2$, the terms involving the partial derivatives of $f$ are 
    $$
    \frac{\partial f_i}{\partial x_l} \cdot \frac{\partial f_j}{\partial x_l} \text{  or  } \frac{\partial^2 f_i}{\partial x_l^2}, ~\forall i,j=1,2,\dots,n.
    $$
    Note the first type of terms are bounded by $\lambda_2(f)^{1/2} \cdot \lambda_2(f)^{1/2}$ and the second type of terms are bounded by $\lambda_2(f)$. If we take
    $c=12$, $|\partial_l^2 g \circ f| \leq c C_2(g)\lambda_2(f)$. On the other hand, 
    \begin{equation}
        \label{eq:cauchy_schwarz}
        |\partial_l g \circ f|^2 \leq \left[\sum_{i=1}^3 \frac{\partial g}{\partial f_i} \frac{\partial f_i}{\partial x_l}\right]^2 
        \leq [3C_2(g)^{\frac{1}{2}}\lambda_2(f)^{\frac{1}{2}}]^2 \leq 9 C_2(g)\lambda_2(f).
    \end{equation}
    For $r=3$, an equation similar to \eqref{eq:cauchy_schwarz} will give us $|\partial_l g \circ f|^3 \leq 27 C_3(g)\lambda_3(f)$. Meanwhile,
    $$
    \begin{aligned}
        |\partial_l^2 g \circ f|^{\frac{3}{2}} &= \left[\sum_{i,j=1}^3\left(\frac{\partial^2 g}{\partial f_i \partial f_j}\right)
            \left(\frac{\partial f_i}{\partial x_l} \frac{\partial f_j}{\partial x_l}\right)
            +\sum_{i=1}^3\left(\frac{\partial g}{\partial f_i}\right) \left(\frac{\partial^2 f_i}{\partial x_l^2}\right)\right]^{\frac{3}{2}}\\
            &\leq \left[9 C_3(g)^{\frac{2}{3}} (\lambda_3(f)^{\frac{1}{3}} \cdot \lambda_3(f)^{\frac{1}{3}}) + 3 C_3(g)^{\frac{1}{3}} (\lambda_3(f)^{\frac{2}{3}})\right]^{\frac{3}{2}} \\
            &= 12^{\frac{3}{2}} C_3(g)\lambda_3(f).
    \end{aligned}
    $$
    For the third derivatives $\partial_l^3 g \circ f$, the terms that involve $f$ are,
    $$
    \frac{\partial f_i}{\partial x_l} \cdot \frac{\partial f_j}{\partial x_l} \cdot \frac{\partial f_k}{\partial x_l} 
    \text{  or  }
    \frac{\partial^2 f_i}{\partial x_l^2} \cdot \frac{\partial f_j}{\partial x_l}
    \text{  or  }
    \frac{\partial^3 f_i}{\partial x_l^3}
    $$
    which are all bounded by $\lambda_3(f)$ and therefore $\lambda_3(g\circ f) \leq 57 C_3(g)\lambda_3(f)$. In summary, we can take $c=57$.
\end{proof}

\section{Existence of smooth approximation $\tP$}
We prove the existence of such functions as claimed in the proof of Theorem \ref{thm:pivot_conv}.
Define $P(x, a, b) = P(x; \sigma^2, m, a, b)$. We first prove the following lemma,

\begin{lemma}
  Define
$$
D(\delta) = \left\{(x,a,b): a \leq x \leq b, b-a \geq \min \left(\delta, \frac{1}{\min(|b|,|a|)}\right) \right\}.
$$
Then, on $D(\delta)$ for any $\delta < 1/4$ we have
$$
\max\left(\frac{e^{-x^2/2}}{\Phi(b)-\Phi(a)}, \frac{e^{-a^2/2}}{\Phi(b)-\Phi(a)}, \frac{e^{-b^2/2}}{\Phi(b)-\Phi(a)} \right) \leq C \max\left(\delta^{-1}, \min(|a|,|b|)\right)
$$
for some universal constant $C$.
\end{lemma}

\begin{proof}
Note that for any $\delta > 0$  on 
$$D(\delta) \cap \left\{(x,a,b): \sign(a)=\sign(b) \right\}$$ we have
$$
e^{-x^2/2} \leq \max(e^{-a^2/2}, e^{-b^2/2}) 
$$
so it suffices to prove that
$$
\max\left(\frac{e^{-a^2/2}}{\Phi(b)-\Phi(a)}, \frac{e^{-b^2/2}}{\Phi(b)-\Phi(a)} \right) \leq C \delta^{-1}
$$
on $D(\delta)$ for $\delta  < 1/4$ as well as
$$
\sup_{(x,a,b) \in D(\delta) \cap \left\{(x,a,b): \sign(a) \neq \sign(b) \right\}} \frac{e^{-x^2/2}}{\Phi(b)-\Phi(a)} \leq C \delta^{-1}.
$$

Let's consider this latter case first. For any $\delta > 0$  on the set
$D(\delta) \cap \left\{(x,a,b): \sign(a) \neq \sign(b) \right\}$ 
we have
$$
\begin{aligned}
\frac{e^{-x^2/2}}{\Phi(b)-\Phi(a)} & \leq \frac{1}{\Phi(b)-\Phi(a)} \\
& \leq \frac{1}{\inf_{(a,b): \sign(a) \neq \sign(b), b - a \geq \delta} \Phi(b)-\Phi(a)} \\
& \leq \frac{\sqrt{2\pi}}{\delta e^{-\delta^2/2}}. \\ 
\end{aligned}
$$

Continuing, we further split the first case into two cases, i.e.
$D(\delta) \cap \left\{(x,a,b): \sign(a) \neq \sign(b) \right\}$ and $D(\delta) \cap \left\{(x,a,b): \sign(a) = \sign(b) \right\}$. For the first part, analogous to the analysis above, we have
$$
\begin{aligned}
\max\left( \frac{e^{-a^2/2}}{\Phi(b)-\Phi(a)}, \frac{e^{-b^2/2}}{\Phi(b)-\Phi(a)} \right) & \leq \frac{1}{\Phi(b)-\Phi(a)} \\
& \leq \frac{1}{\inf_{(a,b): \sign(a) \neq \sign(b), b - a \geq \delta} \Phi(b)-\Phi(a)} \\
& \leq \frac{\sqrt{2\pi}}{\delta e^{-\delta^2/2}}. \\ 
\end{aligned}
$$

Now, we reduce to the case $D(\delta) \cap \left\{(x,a,b): \sign(a) = \sign(b) \right\}$ and without loss
of generality, we consider the case where $0 < a < b$.
Note that for any $\delta > 0$ on $D(\delta) \cap \left\{(x,a,b): \sign(a) = \sign(b) \right\} \cap \left\{ 0 < a < b\right\}$
$$
\begin{aligned}
\max\left(\frac{e^{-a^2/2}}{\Phi(b)-\Phi(a)}, \frac{e^{-b^2/2}}{\Phi(b)-\Phi(a)}\right)& \leq
\sup_{a: a < 1/\delta}  \frac{e^{-a^2/2}}{\Phi(a+\delta)-\Phi(a)},\\
\max\left(\frac{e^{-a^2/2}}{\Phi(b)-\Phi(a)}, \frac{e^{-b^2/2}}{\Phi(b)-\Phi(a)}\right)& \leq
\sup_{a: a > 1/\delta}  \frac{e^{-a^2/2}}{\Phi(a+1/a)-\Phi(a)},
\end{aligned}
$$

For $\delta < 1/4$ and $0 < a < 1/\delta$, 
$$
\begin{aligned}
(2\pi)^{1/2}(\Phi(a+\delta)-\Phi(a)) & \geq \delta e^{-(a+\delta)^2/2} \\
&= \delta e^{-a^2/2} e^{-a\delta-\delta^2/2} \\
\end{aligned}
$$
Or,
$$
\begin{aligned}
\frac{e^{-a^2/2}}{\Phi(a+\delta)-\Phi(a)}  & \leq  \delta^{-1} (2\pi)^{1/2} e^{a\delta+\delta^2/2} \\
& \leq C_1 \delta^{-1} ,
\end{aligned}
$$
where $C_1 = (2\pi)^{1/2} e^{1 + 1/32}$.

Similarly, for $0 < 1/\delta < a$, 
$$
\begin{aligned}
(2\pi)^{1/2}(\Phi(a+1/a)-\Phi(a)) & \geq \frac{1}{a} e^{-(a+1/a)^2/2} \\
&= \frac{1}{a} e^{-a^2/2} e^{-1- 1/(2 a^2)} \\
\end{aligned}
$$
Or,
$$
\begin{aligned}
\frac{e^{-a^2/2}}{\Phi(a+1/a)-\Phi(a)}  & \leq  a (2\pi)^{1/2} e^{1+1/(2a^2)} \\
& \leq a (2\pi)^{1/2} e^{1+1/(2a^2)} \\
& \leq a (2\pi)^{1/2} e^{1+\frac{1}{2}\delta^2} \\
& \leq C_1 a
\end{aligned}
$$
where $C_1 = (2\pi)^{1/2} e^{1 + 1/32}$. Therefore, we have on $D(\delta) \cap \left\{(x,a,b): \sign(a) = \sign(b) \right\} \cap \left\{ 0 < a < b\right\}$, 
$$
\max\left(\frac{e^{-a^2/2}}{\Phi(b)-\Phi(a)}, \frac{e^{-b^2/2}}{\Phi(b)-\Phi(a)}\right) \leq \max\left(\frac{1}{\delta}, \min(|a|, |b|)\right).
$$
\end{proof}

\begin{remark}
    If we take 
    $$
    D(\delta) = \left\{(x,a,b): a \leq x \leq b, b-a \geq \delta, \min(|a|, |b|) \leq 1/\delta \right\},
    $$
    then
    $$
    \max\left(\frac{e^{-x^2/2}}{\Phi(b)-\Phi(a)}, \frac{e^{-a^2/2}}{\Phi(b)-\Phi(a)}, \frac{e^{-b^2/2}}{\Phi(b)-\Phi(a)} \right) \leq C \delta^{-1}, 
    $$
    where $C$ is a universal constant.
\end{remark}

\begin{corollary}
    \label{cor:derivatives}
For $\delta < 1/4$ and any multi-index $\alpha=(\alpha_1, \alpha_2, \alpha_3)$ we have 
$$
\sup_{(x,a,b) \in D(\delta)} C_{|\alpha|}(P) \leq K_{|\alpha|} \delta^{-|\alpha|}.
$$
for constants $K_l, l \geq 1.$
\end{corollary}

\begin{proof}
    We prove for $\alpha = (0,0,1)$, and similar proofs can be extend to other multi-index $\alpha$ as well.
    Since $P(x, a, b) = 2\min(F(x; a, b), 1 - F(x;a,b))$, we only need to prove for $F(x; a, b)$.
    $$
    \dfrac{\partial F}{\partial b} = -\dfrac{\Phi(x) - \Phi(a)}{[\Phi(b) - \Phi(a)]^2} \cdot \dfrac{1}{\sqrt{2 \pi}} 
    \cdot \exp(-b^2/2). 
    $$
    Therefore,
    $$
    \bigg| \dfrac{\partial F}{\partial b} \bigg| \leq \frac{1}{\sqrt{2 \pi}} \cdot
    \bigg|\dfrac{\Phi(x) - \Phi(a)}{\Phi(b) - \Phi(a)}\bigg| \cdot \dfrac{\exp(-b^2/2)}{\Phi(b) - \Phi(a)}
    \leq C \frac{1}{\delta}.
    $$
\end{proof}

Finally, we put the lemma and the corollary together and prove the following lemma.
\begin{lemma}
    \label{lem:existence}
    There exists a thrice differentible approximation $\tP$ to $P$ that satisfies,
    \begin{itemize}
        \item $\tP(x, a, b)$ is supported on $\{(x,a,b): a \leq x \leq b\}$,
        \item $C_3(\tP) \leq K_3 \dfrac{1}{\delta^6}$ on the set $D(\delta)$,
        \item $\norm{\tP}_{\infty} \leq (K_1 + 1)\delta$ on the set $D(\delta)$.
    \end{itemize}
\end{lemma}

\begin{proof}
    Let $P_{\delta}$ be the smoothed version of $P$ for the minimum function in $P = 2 \min(F, 1-F)$,
    and $\norm{P_{\delta} - P}_{\infty} \leq \delta$.
    Let $\tP_{\delta} = P_{\delta} I_{\delta^2}(x,a,b)$, where $I_{\delta^2}(x,a,b)$ is the smoothed
    version of the indicator function on $\{a \leq x \leq b\}$. $I_{\delta^2}(x,a,b)$ also satifies
    the condition that
    \[
        I_{\delta^2}(x, a, b) = 
        \begin{cases}
        0 & x < a, \text{ or } x > b, \\
        1 & a + \delta^2 \leq x \leq b - \delta^2, \\
        [0,1] & \text{else}.
        \end{cases}
        \]
    $I_{\delta^2}(x, a, b)$ also satisfies $C_3(I_{\delta^2}) \leq \dfrac{1}{\delta^6}$, for some
    universal constant $C$. Thus it is not hard to verify that $C_3(\tP_{\delta}) \leq K_3 \frac{1}{\delta^6}$.
    $$
    \begin{aligned}
        \norm{\tP_{\delta} - P}_{\infty} & \leq \norm{P - P I_{\delta^2}}_{\infty} + \norm{P_{\delta} I_{\delta^2} - P I_{\delta^2}}_{\infty}\\
        & \leq \sup_{a \leq x \leq a + \delta^2, \text{ or } b - \delta^2 \leq x \leq b}\|P\|_{\infty} + \delta \\
        & \leq C_1(P) \cdot \delta^2 + \delta \\
        & \leq (K_1 + 1) \delta.
    \end{aligned}
    $$
\end{proof}

\section{LASSO related proofs}

\subsection{Proof of Lemma \ref{lem:sparsity}}

 We first introduce the following Lemma
in \citet{negahban_unified_2012}. 
\begin{lemma}
\label{lem:l2norm}
If we assume the same assumptions and notations as in Lemma \ref{lem:sparsity}, then
with probability at least $1 - c_1 \exp(-c_1\lambda^2)$, the following
two inequalities hold:
\begin{align} 
\lambda &\geq 2\|X\epsilon\|_{\infty}, \label{eq:lambda} \\
\|\hat{\beta} - \beta^0\|_2 &\leq \frac{8}{m} \sqrt{\frac{k \log p}{n}}, \label{eq:l2norm}
\end{align} 
where $k$ is the number of nonzero entries in $\beta^0$ and $\hat{\beta}$ is
the solution to \eqref{eq:Lasso} 
\end{lemma}

Proof of Lemma \ref{lem:sparsity},
\begin{proof}
According to the KKT conditions, 
\[
\begin{cases}
x_j^T (y - X \hbeta) = \lambda \text{sign}(\hbeta), & \quad \text{if} \quad j \in \A, \\
|x_j^T (y - X \hbeta)| \leq \lambda & \quad \text{if} \quad j \not\in \A.
\end{cases}
\]

According to Lemma \ref{lem:l2norm}, we assume both \eqref{eq:lambda} and \eqref{eq:l2norm}
hold. This happens with probability $1 - c_1 \exp(-c_1\lambda^2)$. For any $j$, 
\begin{align*}
x_j^T(y - X \hbeta) &= x_j^T(X \beta^0 - X \hbeta + \epsilon) \\
& = x_j^T X (\beta^0 - \hbeta) + x_j^T \epsilon \\
& \geq x_j^T X (\beta^0 - \hbeta) - \frac{\lambda}{2}.
\end{align*}

Thus for $j \in \A$,
\begin{align*}
|x_j^T X (\hbeta - \beta^0)| &\geq \frac{\lambda}{2}, \\
|x_j^T X (\hbeta - \beta^0)| &\leq \frac{3\lambda}{2}. \\
\end{align*}

\begin{align*}
\|X^T X(\hbeta - \beta^0)\|_2^2 & = \sum_{j \in \A} \left(x_j X(\hbeta - \beta^0)\right)^2
+ \sum_{j \not\in \A} \left(x_j X(\hbeta - \beta^0)\right)^2 \\
& \geq \sum_{j \in \A} \left(x_j X(\hbeta - \beta^0)\right)^2 \\
& \geq \frac{\lambda^2}{4} |\text{supp}(\hbeta)|.
\end{align*}

Also,
\begin{align*}
\|X^T X(\hbeta - \beta^0)\|_2^2 &\leq \|X^T X\|_2^2 \|\hbeta - \beta^0\|_2^2 \\
& \leq \frac{1}{n^2}\|X\|_2^2 \cdot \frac{4k \lambda^2}{m^2} \\
& \leq \phi_{max}^2 \frac{4k \lambda^2}{m^2}.
\end{align*}

Combining the two inequalities, we have that
\begin{equation*}
|\text{supp}(\hbeta)| \leq \frac{16 k\phi_{max}^2}{m^2}
\end{equation*}
holds with probability $1 - c_1 \exp(-c_1\lambda^2)$. 
\end{proof}

\subsection{Proof of Lemma \ref{lem:eigenvalues}}

\begin{proof}
Note $X_{E}^{\dagger} = (X_{E}^T X_{E})^{-1} X_{E}^T$. According to the 
 assumption assumed in Lemma \ref{lem:eigenvalues}, $\phi_{\text{min}} > \nu$. Thus
for any possible active set $E$,
$$
\max_{i,j} |\left[(X_{E}^T X_{E})^{-1}\right]_{ij}| \leq 1/\nu^2.
$$ 
The above result can be easily obtained using Singular Value Decomposition on $X_{E}$.
Therefore, we have
$$
\max_{i,j}|X_{E}^{\dagger}|_{ij} \leq |E|/\nu^2 \max_{i,j} |X_{ij}| 
$$
\end{proof}

\subsection{Proof of Lemma \ref{lem:lower_upper}}

\begin{proof}
Without loss of generality, we assume $\beta^0 = 0$.
We first see that for any fixed $E$, and $\eta_n$ chosen as in 
\eqref{eq:contrast}, the upper and lower bound simplies to 
$$
\upperboundfix = \infty, \quad \lowerboundfix = \max_{k \in E, k \neq j} 
\frac{\lambda\cdot \tau[1 + (|E|-1)\rho] - \bar{\beta}_{k,E}}{\rho} + \bar{\beta}_{j,E},
$$
where $\bar{\beta}_E \in \real^{|E|}$ is the least square estimator with the $E$ variables, 
$$
\bar{\beta}_E = (X_E^T X_E)^{-1} X_E^T y.
$$
Note that the first two equations of \eqref{eq:upper_lower} are automatically satisfied in
this case. Without loss of generality, we assume $L_{\E^*} > 0$, and noticing 
$L_{\E^*} \leq \max_{\E} L_{\E}$, we have
$$
\Pp(\lowerbound(y_n) > 1/\delta_n) \leq \Pp(\max_{\E} \lowerboundfix(y_n) > 1/\delta_n).
$$
Since $\max_{\E} \lowerboundfix(y_n)$ is the maximum of at most $p^{cK}$ sub-Gaussian variables,
thus the RHS is bounded by $O(e^{-\lambda_n^{\kappa}}) = O(p^{-\kappa})$, which goes to $0$.
\end{proof}

\subsection{Proof of Lemma \ref{lem:lasso:conv}}

\begin{proof}
Per Lemma \ref{lem:sparsity}, we have 
$$
|E_n| \leq cK, \qquad |{\cal S}_n| \leq p^{cK},
$$ 
with probability at least $1 - c_1\exp(-c_1\lambda_n^2)$.

We modify the selection procedure $\E^*_n$ on the small probability
event. More specifically, we define $\wE^*_n$ as
$$
\wE^*_n(y_n, X_n) = 
\begin{cases}
\E^*_n(y_n, X_n), &\text{if   } |E_n| \leq cK,\\
\text{no selection}, &\text{else.}
\end{cases}
$$
It is easy to see that $\wE^*_n$ is also an affine selection procedure.
which differs from $\E^*_n$ only on the event $\{|E_n| > cK\}$. Thus the 
pivots formed with $\E^*_n$ and $\wE^*_n$ converge in probability, 
$$
\begin{aligned}
\bigg|\Pp\big[P(\eta_n^Ty_n&; \eta_n^{T} \Sigma_n \eta_n, \eta_n^T\mu_n, L_{\E^*_n}, U_{\E^*_n}) \\
&\neq P(\eta_n^Ty_n; \eta_n^{T} \Sigma_n \eta_n, \eta_n^T\mu_n, L_{\wE^*_n}, U_{\wE^*_n})\big]\bigg| \leq \frac{c_1}{p^{c_1}}  \rightarrow 0.
\end{aligned}
$$

Therefore, we only need to consider the asymptotic distribution of 
the pivot with $\wE^*_n$ as the selection procedure.
Note that for $\wE^*_n$, 
$$
M(\wE^*_n, \eta_n) \leq \frac{cK}{\nu^2} \cdot \max|X_{ij}|,\quad \nrow(\wE^*_n) \leq p,
\quad |\mathcal{S}_n| \leq p^{cK}.
$$
Now with our choice of $\delta_n$'s, it is easy to rewrite the condition in Theorem
\ref{thm:pivot_conv} as 
$$
n^{-1/2}(\log p_n)^{7+3\kappa} \rightarrow 0.
$$
\end{proof}

\end{document}